\documentclass[10pt,a4paper,twoside,reqno]{amsart}

\usepackage{amssymb,amsmath,latexsym,amsthm}
\usepackage{mathrsfs} 
\usepackage{url}
\usepackage[english]{babel}

\newtheorem{thm}{Theorem}[section]

\newtheorem{lemma}[thm]{Lemma}

\theoremstyle{definition}

\numberwithin{equation}{section}

\def\Ocal{\mathcal{O}}
\def\Ascr{\mathscr{A}}

\begin{document}

\title[Products of primes in arithmetic progressions]{Products of primes in arithmetic progressions: a footnote in parity breaking}

\author{Olivier Ramar\'e}
\address{Olivier Ramar\'e\\
CNRS / Institut de Math\'ematiques de Marseille\\
Aix Marseille Universit\'e, U.M.R. 7373\\
Site Sud, Campus de Luminy, Case 907\\
13288 MARSEILLE Cedex 9, France}

\email{olivier.ramare@univ-amu.fr}
\urladdr{http://iml.univ-mrs.fr/~ramare/}

\author{Aled Walker}
\address{Aled Walker\\
Mathematical Institute\\
University of Oxford\\
Andrew Wiles Building\\
Radcliffe Observatory Quarter\\
Woodstock Road\\
Oxford\\
OX2 6GG, United Kingdom}

\email{walker@maths.ox.ac.uk}
\urladdr{https://www.maths.ox.ac.uk/people/aled.walker}



\thanks{The first author has been partly supported by the Indo-French
  Centre for the Promotion of Advanced Research -- CEFIPRA, project No
  5401-1. The second author was supported by the EPSRC Grant EP/M50659X/1.}

\begin{abstract}
  We prove that, if $x$ and $q\leqslant x^{1/16}$ are two
  parameters, then for any invertible residue class $a$ modulo $q$
  there exists a product of exactly three primes, each one below
  $x^{1/3}$, that is congruent to $a$ modulo $q$.
\end{abstract}

\maketitle

\bigskip
\section{Introduction and results}

Xylouris' version of Linnik's Theorem \cite{Xylouris*11} tells us
that, for every modulus $q$ and every invertible residue class $a$
modulo~$q$, one can find a prime congruent to $a$ modulo~$q$ that is
below $q^{5.18}$ \emph{provided} $q$ be large enough. The proof relies
on intricate techniques, and though the result is indeed effective, no
one has been able to give any explicit version of it. The aim of this
paper is to show that one can easily access a fully explicit result,
with respectable constants, provided one replaces primes by products
of three primes. Here is what we prove, by combining a simple sieve
technique together with classical additive combinatorics.

\begin{thm}
  \label{main}
  Let $x$ and $q\le x^{1/16}$ be two parameters. Then for any
  invertible residue class $a$ modulo~$q$, there exists a product of
  three primes, all below $x^{1/3}$, that is congruent to $a$ modulo~$q$.
\end{thm}

We did not try to be optimal in our treatment but sought the simplest
argument. The main surprise is that we use sieve techniques in the
form of Brun-Titchmarsh inequality but we are not blocked by the
parity principle. The reader may argue that we use a lower bound for
$L(1,\chi)$, but the bound we employ is the weakest possible and does
not rely on Siegel's Theorem. In particular, it is not strong enough
to push a possible Siegel zero away from~1, a fact known to be
equivalent to the parity phenomenom (see \cite{Motohashi*79} and \cite{Ramachandra-Sankaranarayanan-Srinivas*96}, or
\cite[Chapter 6]{Ramare*06} for a more complete discussion).

Our theorem is also linked with a conjecture of Erd\"os that says that
every invertible congruence class should contain a product of two
primes not more than $q$. This is discussed in work of the second
author \cite{Walker*16}.

A numerically improved version is being prepared.
\section{Lemmas}

We begin with some crude bounds. 
Let us define
\begin{equation}
  \label{deff0}
  f_0(q)=\prod_{p|q}(1-1/\sqrt{p})^{-1}.
\end{equation}
\begin{lemma}
  \label{boundf0}
  For $q\geqslant 2$ we have 
  $
    f_0(q)\le 3.32 \sqrt{q}
  $
\end{lemma}

\begin{proof}
For all primes $p$ we have $(1-1/\sqrt{p})^{-1}\leqslant \alpha_p\sqrt{p}$, where $$\alpha_p = \begin{cases} 
\frac{1}{\sqrt{2}-1} & p=2 \\
\frac{1}{\sqrt{3}-1} & p=3 \\
1 & \text{otherwise,}
\end{cases}$$
\noindent and since $\alpha_2\leqslant 2.42$ and $\alpha_3\leqslant 1.37$, we obtain the inequalities 

\begin{equation*}
f_0(q)\leqslant 2.42\cdot 1.37\cdot \sqrt{q} \leqslant 3.32\cdot \sqrt{q}.
\end{equation*}
\end{proof}

\noindent We will also require a rudimentary estimate on $\phi(q)$. 

\begin{lemma}
  \label{boundforphi}
  If $q\geqslant 31$ then $\phi(q)> 8$. 
\end{lemma}

\begin{proof}
Recall that $$\phi(n) = n\prod\limits_{p\vert n} \left(1-\frac{1}{p}\right).$$ Therefore if $\phi(q)\leqslant 8$, the only prime factors of $q$ are $2,3,5,7$. By performing an easy case analysis on which of these primes divides $q$, one sees that the only $q$ for which $\phi(q)\leqslant 8$ are $1,2,3,4,5,6,7,8,9,10,12,15,20,24$, and~30.   
\end{proof}

\noindent
We will use the elementary theory of Dirichlet characters, referring the reader to the
excellent monograph \cite{Davenport} of Davenport for an
introduction on the subject. In particular we note the following easy bound.
\begin{lemma}
  \label{boundchi}
  Let $\chi$ be a non-principal Dirichlet character modulo~$q$. Let
  $I$ be a subset of $\{1,\cdots,q\}$. We have
  \begin{equation*}
    \Bigl|\sum_{n\in I}\chi(n)\Bigr|\le \phi(q)/2.
  \end{equation*}
  The same bound holds true for any finite interval instead of~$I$.
\end{lemma}

\begin{proof}
  We know on the one hand that, we have
  $\sum_{1\le n\le q}\chi(n)=0$ by orthogonality, and on the other hand that $\chi(n)$
  does not vanish only when $n$ belongs to the multiplicative group,
  say $\mathcal{U}_q$, of $\mathbb{Z}/q\mathbb{Z}$. We can hence bound
  $|\sum_{n\in I}\chi(n)|$ by the cardinal of
  $I\cap\mathcal{U}_q$ and by the cardinal of $\mathcal{U}_q\setminus
  I\cap\mathcal{U}_q$. One of them is not more than $\phi(q)/2$,
  proving the first part of the lemma. When $I$ is a finite
  interval, we note that the sum of the values of $\chi(n)$ on any
  $q$~consecutive integers vanishes, reducing the problem to the first case.
\end{proof}

We next modify an idea of Gel'fond from \cite{Gelfond*56}, which is maybe
more easily read in~\cite{Gelfond-Linnik*65}. 
\begin{lemma}
  Let $\chi$ be a non-principal quadratic character modulo~$q$. We have
  \begin{equation*}
    L(1,\chi)\ge \frac{\pi}{4\phi(q)}-\frac{\pi}{\phi(q)^2}.
  \end{equation*}
\end{lemma}

\begin{proof}
  We consider the sum $S(\alpha)=\sum_{n\ge1}(1\star
  \chi)(n)e^{-n\alpha}$ for real positive $\alpha$. Since $(1\star
  \chi)(m^2)\ge1$ for every integer $m$, and $(1\star
  \chi)(n)\ge0$ in general, a comparison with an integral gives us
  \begin{equation*}
    1+S(\alpha)\ge \sum_{m\ge0}e^{-m^2\alpha}
    \ge \int_0^\infty e^{-\alpha t^2}dt
    = \frac{\Gamma(1/2)}{2\sqrt{\alpha}}
    = \frac{\sqrt{\pi}}{2\sqrt{\alpha}}.
  \end{equation*}
  On the other hand we can expand $(1\star
  \chi)(n)=\sum_{d|n}\chi(d)$ and get
  \begin{equation*}
    S(\alpha)=\sum_{d\ge1}\frac{\chi(d)}{e^{\alpha d}-1}
    =
    \frac{L(1,\chi)}{\alpha}-\sum_{d\ge1}\chi(d) g(\alpha d)
  \end{equation*}
  by using the non-negative non-increasing function
  $g(x)=\frac{1}{x}-\frac{1}{e^x-1}$. We  find that, by Lemma~\ref{boundchi},
  \begin{align*}
   \sum_{d\ge1}\chi(d) g(\alpha d)
     &=
    -\sum_{d\ge1}\chi(d)\int_{\alpha d}^\infty  g'(t)dt
     \\&=
     -\int_0^\infty \sum_{d\le t/\alpha }\chi(d)g'(t)dt
     \\&\ge
     \frac{\phi(q)}{2}\int_0^\infty g'(t)dt
         = -{\phi(q)}/{4}
  \end{align*}
  since $\operatorname{lim}g(x)=1/2$ as $x$ tends to $0$ from above.
  By comparing both upper and lower estimate
  for $S(\alpha)$, we reach
  \begin{equation*}
    L(1, \chi)\ge \frac{\sqrt{\pi\alpha}}{2}-\alpha-\frac{\alpha \phi(q)}{4}.
  \end{equation*}
  We select $\alpha =
  \pi/\phi(q)^2$. The lemma follows.
\end{proof}


\begin{lemma}
  \label{linnik}
  Let $q\ge3$ be an integer and $\chi$ be a non-principal quadratic
  character modulo~$q$. Then there is a prime $p$ at most $q^4$ such
  that $\chi(p)=1$.
\end{lemma}

\begin{proof} 
  We adapt the proof of J. Pintz taken from~\cite{Pintz*77-2}. Assume
  that no primes not more than a given real number $x$ are in the
  kernel of $\chi$. We use the notation $d|q^\infty$ to say that all
  the prime factors of $d$ divides $q$. Then on the one side we have
  \begin{equation*}
    \sum_{n\le x}(1\star\chi)(n)
    =\sum_{d|q^\infty}\sum_{\substack{m^2\le x/d,\\ (m,q)=1}}1
    \le\sum_{d|q^\infty}\sqrt{\frac{x}{d}}\le \sqrt{x}f_0(q)
  \end{equation*}
  where $f_0$ is the function defined in~\eqref{deff0},
while on the other side we can approximate this sum by $L(1,\chi)$ as follows:
\begin{equation*}
  \sum_{n\le x}(1\star\chi)(n)
  =
  \sum_{d\le x}\chi(d)\Bigl[\frac{x}{d}\Bigr]
  =x\sum_{d\le x}\frac{\chi(d)}{d}
  -
  \sum_{d\le x}\chi(d)\Bigl\{\frac{x}{d}\Bigr\}.
\end{equation*}
The first summation over $d$ is an approximation of $L(1,\chi)$
(recall Lem\-ma~\ref{boundchi}):
\begin{align*}
  L(1,\chi)
  &=
  \sum_{d\ge1}\frac{\chi(d)}{d}
  =
  \sum_{d\le x}\frac{\chi(d)}{d}
  +\int_{x}^\infty\sum_{x<d\le t}\chi(d)dt/t^2
  \\&=
  \sum_{d\le x}\frac{\chi(d)}{d}
  +\Ocal^*\Bigl(\frac{\phi(q)}{2x}\Bigr).
\end{align*}
We treat the second summation in $d$ above by Axer's method
from~\cite{Axer*11} (see also \cite[Theorem 8.1]{Montgomery-Vaughan*06}):
\begin{equation*}
  \Bigl|\sum_{d\le x}\chi(d)\Bigl\{\frac{x}{d}\Bigr\}\Bigr|
  \le
  \sum_{d\le y}1
  +
  \sum_{m\le x/y}\Bigl|\sum_{d:[x/d]=m}\chi(d)\Bigl\{\frac{x}{d}\Bigr\}\Bigr|
  \le y + \frac{\phi(q)x}{2y}\le \sqrt{2\phi(q)x}
\end{equation*}
by selecting $y=\sqrt{\phi(q)x/2}$, the second inequality following by
Abel summation. All of this implies that
$ \sqrt{x}L(1,\chi)\le f_0(q)+\sqrt{2\phi(q)}+\phi(q)/(2\sqrt{x})$.
However, the previous lemma gives us a lower bound for $L(1,\chi)$ and
thus we should have
\begin{equation*}
  \frac{\pi}{4\phi(q)}-\frac{\pi}{\phi(q)^2}\le 
  \frac{f_0(q)}{\sqrt{x}}+\sqrt{\frac{2\phi(q)}{x}}
  +\frac{\phi(q)}{2x}
\end{equation*}

\noindent We substitute $x=q^4$, using the upper bound for $f_0(q)$
provided by Lem\-ma~\ref{boundf0}. Replacing the left hand side of the
above inequality by $\pi/8\phi(q)$, which is permissible by
Lemma~\ref{boundforphi}, together with the bound $\phi(q)\le q$, after
a short calculation we derive a contradiction for all $q \ge 45$.
Calculating using the exact expressions for $f_0(q)$ and $\phi(q)$
when $q\in\{15,\cdots, 45\}$, we also derive a contraction. For the
remaining $q$ it is easy enough to find primes $p\leqslant q^4$ such
that $p\equiv 1$ modulo $q$. Indeed, for $q=2,\cdots,14$ we may take
$p=3,7,5,11,7,29,17,19,11,23,37,53,29$ respectively.
\end{proof}

We quote the following result from~\cite{Montgomery-Vaughan*74}, which is a
strong form of the Brun-Titchmarsh inequality.
\begin{lemma}
  \label{BT}
  When $1\le q<x$, we have
  \begin{equation*}
    \sum_{\substack{y<p\le y+x,\\ p\equiv a[q]}}1\le \frac{2x}{\phi(q)\log(x/q)}.
  \end{equation*}
for any positive $y$.
\end{lemma}

\begin{lemma}
  \label{dusart}
  We have $\pi(x)\ge x/(\log x-1)$ when $x\ge 5\,393$. Furthermore the number
  of primes not more than $x$ but prime to some fixed modulus $q$ below $x$ is
  at least $x/\log x$, again when $x\ge 5\,393$.
\end{lemma}

\begin{proof}
  The first inequality is taken from \cite{Dusart*07}. For the second, we
  simply note that the number of prime factors of $q$ is at most $(\log
  x)/\log 2$ and that
  \begin{equation*}
    \frac{x}{\log x-1}-\frac{\log x}{\log 2}\ge \frac{x}{\log x}
  \end{equation*}
  when $x\ge 5\,000$.
\end{proof}

The final ingredient in the argument will be Kneser's Theorem, which we now
recall (see \cite[Theorem 4.3]{Nathanson*96} or~\cite[Theorem 5.5]{Tao-Vu*06}).
\begin{lemma}
  \label{kneser}
  Let $A$ and $B$ be two subsets of the finite abelian group $G$. Let $H$ be
  the subgroup of elements $h$ of $G$ that stabilizes $A+B$, i.e. that are
  such that $h+A+B=A+B$. We have
  \begin{equation*}
    |A+B|\ge |A+H|+|B+H|-|H|.
  \end{equation*}
\end{lemma}

\section{Proof of Theorem~\ref{main}}

Let us first treat the case $x\ge 10^{16}$.

Let $X=x^{1/3}$. Since this parameter is at least $10^5$, Lemma~\ref{dusart}
tells us that the number $\pi_q(X)$ of primes below $X$ which are coprime to
$q$ is at least $X/\log X$.  The Brun-Titchmarsh inequality in the form given
by Montgomery \& Vaughan, recalled in Lemma~\ref{BT}, tells us that the number
of primes less than $X$ in any progression $a\mod q$, for $a$ prime to~$q$, is
at most $\frac{32}{13}X/(\phi(q)\log X)$. This implies, when compared to the
total number of primes coprime to $q$ given by Lemma~\ref{dusart}, that at
least $\frac{13}{32}\phi(q)$ such residue classes contain a prime. Let us call
this set of classes $\Ascr$ and apply Kneser's Theorem (Lemma~\ref{kneser}) to
the group $G$ of invertible residues modulo~$q$. Let $H$ be the stabilizer of
$\Ascr\cdot\Ascr$. We divide into cases according to the index of $H$.

If $H$ is equal to $G$ then, since
$\Ascr\cdot\Ascr\cdot H=\Ascr\cdot\Ascr$, we have $\Ascr\cdot\Ascr=G$
and of course $\Ascr\cdot\Ascr\cdot\Ascr=G$.

If $H$ has index $2$, then it is the kernel of some quadratic
character $\chi$. Because $\Ascr$ generates $G$ multiplicatively,
there is a point $a$ in $\Ascr$ such that $\chi(a)=-1$. By
Lemma~\ref{linnik}, there is another one, say $a'$, such that
$\chi(a')=1$. Hence $\Ascr\cdot\Ascr$ also has a point $b$ such that
$\chi(b)=1$ and one, say $b'$, such that $\chi(b')=-1$. This implies
that $\Ascr\cdot\Ascr\cdot H=G$, i.e. $\Ascr\cdot\Ascr=G$.

When $H$ is of index~3, then $\Ascr\cdot H$ covers at least 2 $H$-cosets
(since $\tfrac{13}{32}>\tfrac13$) and is thus of cardinality at least $2\phi(q)/3$. Kneser's Theorem
ensures that $|\Ascr\cdot\Ascr|\ge \phi(q)$, i.e. that again
$\Ascr\cdot\Ascr=G$.

When $H$ is of index~4, then $\Ascr\cdot H$ covers at least 2
$H$-cosets (since $\tfrac{13}{32}>\tfrac14$) and is thus of
cardinality at least $\phi(q)/2$. By Kneser's
Theorem,
$$|\Ascr\cdot\Ascr|\ge 2|\Ascr\cdot H| - |H|\ge \frac34\phi(q).$$

When $H$ is of index~$Y$ say, with $Y$ at least 5, let us write
$|\Ascr|/\phi(q)=1/U$. The set $\Ascr\cdot H$ is made out of at least $\lceil
Y/U\rceil$ cosets modulo~$H$. Using the same manipulation as above, Kneser's
Theorem ensures that $|\Ascr\cdot\Ascr|/\phi(q)\ge (2\lceil Y/U\rceil-1)/Y$. A
quick computation shows that the minimum of $(2\lceil Y/U\rceil-1)/Y$ when $Y$
ranges $\{5,6,7,8,9\}$ is reached at $Y=7$ and has value
$5/7$. When $Y$ is larger than 10, we directly check that
$(2\lceil Y/U\rceil-1)/Y\ge \frac{2}{U}-\frac1Y\ge
\frac{13}{16}-\frac{1}{10}\ge\frac{7}{10}$.

Combining these final two cases, we have proved that $|\Ascr\cdot\Ascr|\ge
\frac{7}{10}\phi(q)$. Let $b$ be an arbitrary invertible residue class
modulo~$q$. The set $b/\Ascr$ is of cardinality at least $|\Ascr|$ and, since $\frac{13}{32}$ is greater then $\frac{3}{10}$, this is
strictly larger than the size of the complementary set of
$\Ascr\cdot\Ascr$. Therefore these sets have a point in common: there exist
$a$, $a_1$ and $a_2$, all three in $\Ascr$ such that $b/a=a_1a_2$, proving our
theorem in this case.

\bigskip It remains to deal with $ x< 10^{16}$, which is done
by explicit calculation. The inclusion of this addendum was kindly suggested to us by an anonymous referee. Indeed, when $x < 10^{16}$ , the modulus $q$
is restricted to be not more than $10$, implying that only a limited number of
congruence classes are to be looked at. We proceed by hand:

\smallskip\noindent\hspace{2pt}
\begin{minipage}[b]{0.48\linewidth}
\noindent When $q=2$, we only need $x\ge 3^3$.
\par\noindent When $q=3$, we only  need $x\ge 7^3$.
\par\noindent When $q=4$, we only   need $x\ge 5^3$.
\par\noindent When $q=5$, we only   need $x\ge 19^3$.
\par\noindent When $q=6$, we only   need $x\ge 11^3$.
\end{minipage}\hfill%
\begin{minipage}[b]{0.50\linewidth}
\par\noindent When $q=7$, we  only need $x\ge 29^3$.
\par\noindent When $q=8$, we only   need $x\ge 23^3$.
\par\noindent When $q=9$, we only   need $x\ge 23^3$.
\par\noindent When $q=10$, we only   need $x\ge 19^3$.
\end{minipage}
This takes care of the situation when $x\ge 29^3$. However, when $x$
is below~$29^3$, the bound $x^{1/16}$ is less than $2$. This ends the proof of our theorem.

\bibliographystyle{plain}

\end{document}